\documentclass[11 pt,reqno]{amsart}
\usepackage{amsmath}
\usepackage{amssymb}
\usepackage{amsthm}
\usepackage{amsfonts}
\usepackage{color}
\usepackage{graphicx}
\usepackage{graphics}
\usepackage{epsfig}

\usepackage{color}

 \newtheorem{thm}{Theorem}[section]
 
 \newtheorem{lem}[thm]{Lemma}
 \newtheorem{prop}[thm]{Proposition}
 \theoremstyle{definition}
 \newtheorem{defn}[thm]{Definition}
 
 \newtheorem{prob}[thm]{Problem}
 \theoremstyle{remark}
 \newtheorem{rem}[thm]{Remark}

\numberwithin{equation}{section} \numberwithin{figure}{section}

\newcommand{\CC}{{\mathbb C}}
\newcommand{\DD}{{\mathbb D}}
\newcommand{\TT}{{\mathbb T}}

\newcommand{\NN}{{\mathbb N}}

\DeclareMathOperator{\Pol}{{\mathcal P}}
\DeclareMathOperator{\dist}{dist}
\DeclareMathOperator{\Span}{span}

\begin{document}
\bibliographystyle{alpha}

\title[Cyclicity and extremal polynomials]{Cyclicity in Dirichlet-type spaces and extremal polynomials}
\author[B\'en\'eteau]{Catherine B\'en\'eteau}
\address{Department of Mathematics, University of South Florida, 4202 E. Fowler Avenue, Tampa, FL 33620, USA.}
\email{cbenetea@usf.edu}
\author[Condori]{Alberto A. Condori}
\address{Department of Mathematics, Florida Gulf Coast University,
10501 FGCU Boulevard South, Fort Myers, FL 33965-6565, USA.}
\email{acondori@fgcu.edu}
\author[Liaw]{Constanze Liaw}
\address{Department of Mathematics, Baylor University, One Bear Place \#97328, Waco, TX 76798-7328, USA.}
\email{Constanze$\underline{\,\,\,}$Liaw@baylor.edu}
\author[Seco]{Daniel Seco}
\address{Departament de Matem\`atiques, Universitat Aut\`onoma de Barcelona,
08193 Bellaterra, Spain.} \email{dseco@mat.uab.cat}
\author[Sola]{Alan A. Sola}
\address{Statistical Laboratory, Centre for Mathematical Sciences,
University of Cambridge, Wilberforce Road, Cambridge CB3 0WB, UK.}
\email{a.sola@statslab.cam.ac.uk}
\thanks{CL is partially supported by the NSF grant DMS-1261687. DS
is supported by the MEC/MICINN grant MTM-2008-00145. AS acknowledges
support from the EPSRC under grant EP/103372X/1. CB, DS and AS would
like to thank the Institut Mittag-Leffler and the AXA Research Fund
for support while working on this project.  AC, CL and DS would like to thank B. Wick
for organizing the Internet Analysis Seminar at which they were initially introduced to the problem of cyclicity in the Dirichlet space.}
\date{January 17, 2013}

\keywords{Cyclicity, Dirichlet space, optimal approximation.} \subjclass[2010]{Primary: 31C25,
30H05. Secondary: 30B60; 47A16.}
\begin{abstract}
For functions $f$ in Dirichlet-type spaces $D_{\alpha}$, we study how to determine
constructively optimal polynomials $p_n$ that minimize $\|p f-1\|_\alpha$
among all polynomials $p$ of degree at most $n$.  Then we give upper and lower bounds
for the rate of decay of $\|p_{n}f-1\|_{\alpha}$ as $n$ approaches $\infty$.
Further, we study a generalization of a weak version of the
Brown-Shields conjecture and some computational phenomena about the
zeros of optimal polynomials.
\end{abstract}

\maketitle


\section{Introduction}
\subsection{Cyclicity in spaces of analytic functions}
In this paper, we study certain Hilbert spaces of analytic functions in the open unit disk
$\DD$, denoted $D_{\alpha}$ and referred to as \emph{Dirichlet-type
spaces of order $\alpha$}.  For $-\infty<\alpha<\infty$, the space $D_{\alpha}$ consists of
all analytic functions $f\colon\DD\rightarrow \CC$ whose Taylor coefficients in the expansion
\[f(z)=\sum_{k=0}^{\infty}a_kz^k, \quad z \in \DD,\]
satisfy
\begin{equation*}
\|f\|^2_\alpha=\sum_{k=0}^\infty (k+1)^{\alpha}|a_k|^2< \infty.
\end{equation*}
It is easy to see that $D_{\alpha}\subseteq D_{\beta}$ when $\alpha\geq\beta$, and
$f\in D_{\alpha}$ if and only if the derivative $f'\in D_{\alpha-2}$.

Three values of $\alpha$ correspond to spaces that have been studied
extensively and are often defined in terms of integrability:
\begin{itemize}
\item $\alpha=-1$ corresponds to {\it the Bergman space $B$}, consisting of
functions with
\[\int_{\DD}|f(z)|^2dA(z)<\infty, \quad dA(z)=\frac{dxdy}{\pi},\]
\item $\alpha=0$ corresponds to {\it the Hardy space $H^2$}, consisting of functions
with
\[\sup_{0<r<1}\ \frac{1}{2\pi}\int_{-\pi}^{\pi}|f(re^{i\theta})|^2d\theta<\infty,\]
\item and $\alpha=1$ corresponds to the usual {\it Dirichlet space $D$}
of functions $f$ with
\[\int_{\DD}|f'(z)|^2dA(z)<\infty.\]
\end{itemize}
A description similar to that of the Dirichlet space, in terms of an
integral, is possible for the $D_\alpha$ spaces for $\alpha < 2.$
Indeed, $f \in D_{\alpha}$ if and only if

\begin{equation}\label{dalfaf}
D_\alpha (f)= \int_{\DD} |f'(z)|^2 ( 1 - |z|^2)^{1 - \alpha} dA(z) <
\infty.
\end{equation}
This expression can be used to define an equivalent norm for $f \in D_{\alpha}$,
which we use in Section \ref{s-GRADIENT}. We refer the reader to the books
\cite{Durbook}, \cite{DS04} and \cite{HKZ00} for in-depth treatments of Hardy
and Bergman spaces; recent surveys concerning the Dirichlet space $D$ include
\cite{ARSW11} and \cite{Ross06}.

A function $f \in D_{\alpha}$ is said to be \emph{cyclic} in
$D_\alpha$ if the closed subspace generated by polynomial multiples
of $f$,
\[ [f]=\overline{\textrm{span}\{z^kf\colon k=0,1,2,\ldots\}},\]
coincides with $D_{\alpha}$. Note that cyclicity in
$D_\alpha$ implies cyclicity in $D_\beta$ for all $\beta < \alpha$.  The {\it multiplier space}
$M(D_{\alpha})$ consists of analytic functions $\psi$ such that
the induced operator $M_{\psi}\colon f\mapsto \psi f$ maps
$D_{\alpha}$ into itself; such a function $\psi$ is called a
\textit{multiplier}. Thus cyclic functions are precisely those that are
cyclic with respect to the operator $M_z$. Since polynomials
are dense in $D_\alpha$, we have $[1]=D_{\alpha}$. It
is well known (see \cite{BS84}) that an equivalent (and more useful)
condition for the cyclicity of $f$ is that there exist a sequence of
polynomials $\{p_n\}_{n=1}^{\infty}$ such that
\[\|p_n f-1\|_\alpha\rightarrow 0, \quad \textrm{as}\ n\rightarrow \infty.\]
We note that for certain values of $\alpha$, the multiplier
spaces of $D_{\alpha}$ are relatively easy to determine. For
$\alpha\leq 0$, we have $M(D_{\alpha})=H^{\infty}$, and when $\alpha>1$
the multiplier space coincides with $D_{\alpha}$ itself
(see \cite[p.~273]{BS84}).

In general, it is not an easy problem to characterize cyclic
functions in a space of analytic functions. However, a complete
answer to the cyclicity problem for $H^2$ (the case $\alpha=0$) is given by
a theorem of Beurling (see \cite[Chapter 7]{Durbook}): $f$ is cyclic
if and only if $f$ is an outer function. In particular, a cyclic
function $f\in H^2$ cannot vanish in $\DD$.  In the Bergman
space, the situation is considerably more complicated (see
\cite[Chapter 7]{HKZ00}). A common feature of all $D_{\alpha}$
is that cyclic functions have to be non-vanishing in $\DD$. If
$\alpha>1$, to be non-vanishing in the {\it closed}
unit disk, or equivalently,
\[|f(z)|>c>0, \quad z \in \DD,\]
is a necessary and sufficient condition (see \cite{BS84})
for cyclicity.  However,
when $\alpha\leq 1$, functions may still be cyclic if their zero set on the
boundary, that is, the unit circle $\TT$, is not too big. Here, we define the
zero set in an appropriate sense via, for instance, non-tangential limits.

In \cite{BS84}, L. Brown and A.L. Shields studied
the phenomenon of cyclicity in the Dirichlet space. In
particular, they established the following equivalent condition for
cyclicity: $f$ is cyclic in $D_{\alpha}$ if and only if there
exists a sequence of polynomials $\{p_n\}$ such that
\begin{equation}
\sup_n\|p_nf-1\|_{\alpha}<\infty
\label{cyclicitycond:unifbound}
\end{equation}
and, pointwise as $n\rightarrow \infty$,
\begin{equation}
p_n(z)f(z)\rightarrow 1, \quad z \in \DD.
\label{cyclicitycond:pointwise}
\end{equation}
Brown and Shields also obtained a number of partial results towards
a characterization of cyclic vectors in the Dirichlet space $D$.
Their starting point was a result of Beurling, stating that, for any $f \in D$,
the non-tangential limit $f^*(\zeta)=\lim_{z \rightarrow \zeta}f(z)$
exists {\it quasi-everywhere}, that is, outside a set of logarithmic
capacity zero. Brown and Shields proved that if the zeros of $f^*$,
\[\mathcal{Z}(f^*)=\{\zeta \in \TT\colon f^*(\zeta)=0\},\]
form a set of positive logarithmic capacity, then $f$ cannot be
cyclic. On the other hand, they also proved that
$(1-z)^{\beta}$ is cyclic for any $\beta>0$ and showed
that any polynomial without zeros in $\DD$ is cyclic. Hence, they
asked if being outer and having $\textrm{cap}(\mathcal{Z}(f^*))=0$ is
sufficient for $f$ to be cyclic. This problem remains open and is
commonly referred to as the {\it{Brown-Shields conjecture}}; see however
\cite{EKR09} for recent progress by El-Fallah, Kellay, and Ransford,
and for background material. Subsequent to the Brown and Shields paper, Brown and Cohn
showed (see \cite{BC85}) that sets of logarithmic capacity zero do support
zeros of cyclic functions, and later Brown (see \cite{Brown90})
proved that if $f\in D$ is invertible, that is $1/f\in D$, then $f$ is cyclic.
However, there are cyclic functions $f$ for which $1/f\notin D$,
e.g. $f(z)=1-z$.

The problem of cyclicity in $D$ has been addressed in many papers.
An incomplete list includes \cite{HS90}, where sufficient conditions
for cyclicity are given in terms of Bergman-Smirnov exceptional sets; the paper
\cite{EKR06}, where these ideas are developed further, and examples of uncountable
Bergman-Smirnov exceptional sets are found; and \cite{RS92} where multipliers and
invariant subspaces are discussed, leading, for instance, to a proof that
non-vanishing univalent functions in the Dirichlet space are cyclic.

\subsection{Plan of the paper}\label{s-Plan}

In this paper, we set out to improve understanding of cyclicity
by studying certain classes of cyclic functions in detail.  Many of the results
in this paper are variations of the following questions:  Given a cyclic
function $f \in D_{\alpha}$,
can we obtain an explicit sequence of polynomials $\{p_n\}$ such that
\[\|p_nf-1\|_{\alpha}\to 0\;\text{ as }n\to\infty?\]
Can we give an estimate on the rate of decay of these norms as $n
\rightarrow \infty$? What can we say about the approximating
polynomials?

A natural first guess is to take $\{p_n\}$ as the Taylor polynomials
of the function $1/f$. Since $1/f$ is analytic in $\DD$ by the
cyclicity assumption, we have $p_n\rightarrow 1/f$ pointwise, and
hence \eqref{cyclicitycond:pointwise} is satisfied. However, it may
be the case that norm boundedness
in \eqref{cyclicitycond:unifbound} fails.  This is certainly
true for the Taylor polynomials $T_n(1/f)$ in the case $f(z)=1-z$; indeed, $1/f\notin B\supset H^2\supset D$ and a computation
shows that
\[\|T_n(1/f)f-1\|^2_{D}=\|z^{n+1}\|^2_{D}=n+2.\]

Much of the development that follows is motivated by our goal of
finding concrete substitutes for the Taylor polynomials of $1/f$.

\begin{defn}
Let $f \in D_\alpha$. We say that a polynomial $p_n$ of degree at most $n$
is an {\it optimal approximant} of order $n$ to $1/f$ if $p_n$ minimizes
$\|p f-1\|_\alpha$ among all polynomials $p$ of degree at most $n$.  We call
$\|p_{n}f-1\|_{\alpha}$ the \emph{optimal norm} of degree $n$ associated with $f$.
\end{defn}

In other words, $p_{n}$ is an optimal polynomial of order $n$ to $1/f$ if
\[  \|p_{n}f-1\|_{\alpha}=\dist_{D_{\alpha}}(1,f\cdot\Pol_{n}), \]
where $\Pol_{n}$ denotes the space of polynomials of degree at most $n$ and
\[  \dist_{X}(x,A)=\inf\{\|x-a\|_X: a\in A\}    \]
for any normed space $X$, $A\subseteq X$ and $x\in X$.

Notice that, given $f\in D_{\alpha}\setminus\{0\}$, the existence and uniqueness of
an optimal approximant of order $n$ to $1/f$ follows immediately from the fact that
$f\cdot\Pol_{n}$ is a finite dimensional subspace of the Hilbert space $D_{\alpha}$.
Thus, $f$ is cyclic if and only if the optimal approximants $p_n$ of order $n$ to $1/f$
satisfy $\|p_n f -1\|_\alpha\rightarrow 0$ as $n\rightarrow\infty$.  Furthermore, since
$\|p_n f-1\|_\alpha \leq \|f-1\|_\alpha,$ it follows from
\eqref{cyclicitycond:unifbound} and \eqref{cyclicitycond:pointwise} that $f$ is cyclic
if and only if the sequence of optimal approximants $\{p_n\}_{n=1}^{\infty}$ converges pointwise to $1/f$.

In Section \ref{s-GRADIENT}, we describe a constructive approach for
computing the coefficients of the optimal approximant of order $n$
to $1/f$ for a general function $f$. In particular, Theorem
\ref{t-LINEAR} below states that the coefficients of the optimal
approximants can be computed as ratios of determinants of matrices
whose entries can be explicitly computed via the moments of the
derivative of $f$.  When $f$ itself is a polynomial, these matrices
are banded (see Proposition \ref{t-2tdiagonal}).  As a simple
but fundamental example, we compute optimal approximants to the
function $1/f$ when $f(z)=1-z$.

We are also interested in the rate of convergence of optimal norms.
Since optimal norms decay exponentially for any function $f$ such that $1/f$ is
analytic in the closed unit disk, functions that have zeros on the unit circle
are of particular interest.
In Section \ref{OptimalNormBounds}, we examine the question of whether all functions
with no zeros in the open unit disk but with zeros on the boundary, admitting an
analytic continuation to the closed disk, have optimal norm
achieving a similar rate of decay.
In Theorem \ref{OptimalRate}, we prove that this is indeed the case.

In Section \ref{log}, we deal with a generalization to all $D_\alpha$ of a
subproblem of the Brown-Shields conjecture. We ask the question whether a
function $f$ satisfying $f \in D_\alpha$ and $\log f \in D_\alpha$, must
be cyclic in $D_\alpha$. We note that this is true in the simple cases
of $\alpha = 0$ or $\alpha>1$. In Theorem \ref{solvelog} we are able to
answer affirmatively in the case
$\alpha=1$. Then, Theorem \ref{cyclic1} shows that for the case
$\alpha < 1, \alpha \neq 0,$ the same holds with an additional
technical condition. We do not know if this condition is necessary;
however, it is satisfied by a large class of examples, namely, all of
the functions constructed in Brown-Cohn (\cite{BC85}).

We conclude, in Section 5, by presenting some open questions and basic
computations connected to the zero sets $\mathcal{Z}(p_n)$ of the optimal
approximants $p_n$ of $1/f$ for cyclic functions $f$.

\section{Construction of optimal approximants}\label{s-GRADIENT}

The optimal approximants $p_n$ of order $n$ to $1/f$ are determined by the fact that
$p_nf$ is the orthogonal projection of $1$ onto the space $f\cdot\Pol_n$, and hence,
in principle, if $f\in D_{\alpha}\setminus\{0\}$, they can be computed using the
Gram-Schmidt process.  More precisely, once
a basis for $f\cdot\Pol_n$ is chosen, one can construct an orthonormal basis for
$f\cdot\Pol_{n}$ and then compute the coefficients of $p_n$ with respect to this
orthonormal basis.

In this section, we present a simple method which yields the optimal approximants
$p_n$ without the use of the Gram-Schmidt process, for $\alpha <2$. To that end,
we make use of the integral norm of $D_{\alpha}$, namely,
\[  \|f\|_{\alpha}^{2}=|f(0)|^{2}+D_{\alpha}(f),    \]
where $D_\alpha(f)$ was defined in \eqref{dalfaf}.

Recall that we seek an explicit solution to

\emph{Problem 1.}  Let $n\in\NN$.  Given $f\in D_{\alpha}\setminus\{0\}$ such that
$1\notin f\cdot\Pol_{n}$,
\[  \text{ minimize }\quad\|p f-1\|_{\alpha}\;\text{ over }p\in\Pol_{n}.    \]

\subsection{Construction of optimal approximants via determinants}

As mentioned in Section \ref{s-Plan}, there is a unique optimal approximant $p_{n}\in\Pol_{n}$
of order $n$ to $1/f$ that solves Problem 1, that is,
\[  \|p_{n}f-1\|_{\alpha}=\dist_{D_{\alpha}}(1,f\cdot\Pol_{n}). \]
Observe that for any polynomial $p(z) =\sum_{k=0}^{n}c_{k}z^{k}\in\Pol_{n}$,
\begin{align*}
      \|pf-1\|_{\alpha}^{2}
      &=|p(0)f(0)-1|^{2}+\int_{\DD}|(pf)^{\prime}|^{2}d\mu_{\alpha}\\
      &=|p(0)f(0)-1|^{2}+\int_{\DD}\left|\sum_{k=0}^{n}c_{k}(z^{k}f)^{\prime}\right|^{2}
      d\mu_{\alpha},
\end{align*}
where $d\mu_{\alpha}(z)=(1-|z|^{2})^{1-\alpha}\, dA(z)$.
It follows that if the optimal approximant of order $n$ to $1/f$ vanishes at the origin,
then $\|pf-1\|_{\alpha}^{2}$ is minimal if and only if $c_{0}=c_{1}=\ldots=c_{n}=0$.
Consequently, we may assume without loss of generality that the optimal approximant
$p_{n}$ of order $n$ to $1/f$ does not vanish at the origin.  By replacing $f$ with
$p_{n}(0)f$, we may also assume that $p_{n}(0)=1$ because the optimal approximant
of order $n$ to $1/(p_{n}(0)f)$ is $[p_{n}(0)]^{-1}p_{n}$.  Hence, under this latter
assumption, $p_{n}(z) =1+\sum_{k=1}^{n}c_{n}^{*}z^{k}$ is the optimal approximant of
order $n$ to $1/f$ if and only if $(c_1^{*},\ldots,c_n^{*})\in\CC^{n}$ is the unique
solution to

\emph{Problem 2.}  Let $n\in\NN$.  Given $f\in D_{\alpha}\setminus\{0\}$ such that
$1\notin f\cdot\Pol_{n}$,
\[  \text{ minimize }\quad
        \int_{\DD}\left|f^{\prime}+\sum_{k=1}^{n}c_{k}(z^{k}f)^{\prime}\right|^{2}
        d\mu_{\alpha}\text{ over }(c_1,\ldots,c_{n})\in\CC^{n}. \]
It is evident that $(c_1^{*},\ldots,c_n^{*})\in\CC^{n}$ is the unique solution to
Problem 2 if and only if
\[  g=\sum_{k=1}^{n}c_{k}^{*}(z^{k}f)^{\prime}\;\text{ satisfies }\;
        \|f^{\prime}+g\|_{L^{2}(\mu_{\alpha})}=\dist_{L^{2}(\mu_{\alpha})}(f^{\prime},Y),   \]
where $Y=\Span\{(z^{k} f)^{\prime}: 1\leq k\leq n\}$. Equivalently, $f^{\prime}+g$ is
orthogonal to $Y$ with respect to the $L^{2}(\mu_{\alpha})$ inner product; that is,
for each $j$, $1\leq j\leq n$,
\[  \langle -f^{\prime}, (z^{j} f)^{\prime}\rangle_{L^{2}(\mu_{\alpha})}
        =\langle g, (z^{j} f)^{\prime}\rangle_{L^{2}(\mu_{\alpha})}.    \]
Hence, $(c_1^{*},\ldots,c_n^{*})\in\CC^{n}$ is the unique solution to Problem 2 if and
only if it is the solution to the non-homogeneous system of linear equations
\begin{equation}\label{keySystem}
            \sum_{k=1}^{n}c_{k}
            \langle (z^{k} f)^{\prime},(z^{j}f)^{\prime}\rangle_{L^{2}(\mu_{\alpha})}
            =\langle -f^{\prime}, (z^{j} f)^{\prime}\rangle_{L^{2}(\mu_{\alpha})},\quad 1\leq k\leq n,
\end{equation}
with $(c_1,\ldots,c_{n})\in\CC^{n}$.

\begin{thm}\label{t-LINEAR}
            Let $n\in\NN$ and $f\in D_{\alpha}\setminus\{0\}$.  Suppose
            $1\notin f\cdot\Pol_{n}$ and let $M$ denote the $n\times n$ matrix with entries
            $\langle (z^{k} f)^{\prime},(z^{j}f)^{\prime}\rangle_{L^{2}(\mu_{\alpha})}$.  Then
            the unique $p_{n}\in\Pol_{n}$ satisfying
            \[  \|p_{n}f-1\|_{\alpha}=\dist_{D_{\alpha}}(1,f\cdot\Pol_{n})  \]
            is given by
            \begin{equation}\label{pnFormula}
                        p_{n}(z)=p_{n}(0) \left( 1 +\sum_{k=1}^{n} \frac{\det M^{(k)}}{\det M}z^{k}\right),
            \end{equation}
            where $M^{(k)}$ denotes the $n\times n$ matrix obtained from $M$ by replacing the
            $k$th column of $M$ by the column with entries
            $\langle -f^{\prime}, (z^{j} f)^{\prime}\rangle_{L^{2}(\mu_{\alpha})}$, $1\leq j\leq n$.
\end{thm}
\begin{proof}
      As mentioned before, if $p_n$ is the optimal approximant of order $n$ to $f$ and
      $p_{n}(0)\neq 0$, then the optimal approximant of order $n$ to $1/f_n$
      is $[p_{n}(0)]^{-1}p_{n}$, where
      $f_{n}=p_{n}(0)f$.  If $[p_{n}(0)]^{-1}p_{n}(z)
      =1+\sum_{k=1}^{n}c_{k}^{*}z^{k}$, then $(c_1^{*},\ldots,c_n^{*})\in\CC^{n}$ is
      the unique solution to the system in \eqref{keySystem} because
      \[  \langle (z^{k} f_{n})^{\prime},(z^{j}f_{n})^{\prime}\rangle_{L^{2}(\mu_{\alpha})}
                    =|p_{n}(0)|^{2}\langle (z^{k} f)^{\prime},(z^{j}f)^{\prime}\rangle_{L^{2}(\mu_{\alpha})}    \]
      for $0\leq k\leq n$ and $1\leq j\leq n$.  It follows now that the $n\times n$ matrix
      $M$ with entries $\langle (z^{k} f)^{\prime},(z^{j}f)^{\prime}\rangle_{L^{2}(\mu_{\alpha})}$
      has non-zero determinant and thus
      \[  c_{k}^{*}=\frac{\det M^{(k)}}{\det M},\quad 1\leq k\leq n,  \]
      by Cramer's rule, where $M^{(k)}$ denotes the $n\times n$ matrix obtained from $M$
      by replacing the $k$th column of $M$ by the column with entries
      $\langle -f^{\prime}, (z^{j} f)^{\prime}\rangle_{L^{2}(\mu_{\alpha})}$,
      $1\leq j\leq n$.  Hence $p_{n}$ is given by \eqref{pnFormula}.
\end{proof}

If $f$ is a polynomial, then the computation of the determinants appearing in
\eqref{pnFormula} can be simplified in view of the following proposition.

\begin{prop}\label{t-2tdiagonal}
Suppose $f$ is a polynomial of degree $t$.  Then the matrix $M$ in Theorem \ref{t-LINEAR}
is banded and has bandwidth at most $2t+1$.
\end{prop}
\begin{proof}
The orthogonality of $z^l$ and $z^m$ for $l \neq m$ (under the $L^{2}(\mu_{\alpha})$
inner product) implies that the $(j,k)$-entry of $M$ equals $0$ if the degree of
$(z^{k}f)^{\prime}$ is strictly less than $j-1$ or if the degree of $(z^{j} f)^{\prime}$
is strictly less than $k-1$; that is, $k+t-1<j-1$ or $j+t-1<k-1$.  Therefore, the only
entries of $M$ that do not necessarily vanish are the ones whose indices $j$ and $k$
satisfy $-t\leq j-k\leq t$.  Thus, $M$ is banded and has bandwidth at most $2t+1$.
\end{proof}

\subsection{An explicit example of optimal approximants}\label{optimalPnEx} Now, we calculate explicitly optimal approximants to $1/f$,
where $f(z)=1-z$.  Even though $f$ is a low order polynomial, this
example is already interesting because $f$ is cyclic in $D_\alpha$
for $\alpha \le 1$, even though it is not invertible for any
$\alpha\ge -1$.

We begin with some general computations.  Let $\beta=1-\alpha$.
Then
\begin{equation}
      \|z^{m}\|_{L^{2}(\mu_{\alpha})}^{2}=\int_{0}^{1}u^{1-\alpha}(1-u)^{m}\,du\nonumber =\prod_{\ell=1}^m \frac{\ell}{\ell+\beta}\label{zmL2Norm}
\end{equation}
holds for any non-negative integer $m$.  Therefore, if
$f(z) =\sum_{i=0}^{t}a_{i}z^{i}$, we have\footnote{Under the usual
convention that $a_{i}=0$ for any integer $i<0$ or $i>t$.}
\begin{align}
      \langle (z^{k} f)^{\prime},(z^{j}f)^{\prime}\rangle_{L^{2}(\mu_{\alpha})}
      &=\sum_{i=0}^{t}\sum_{\ell=0}^{t}a_{i}\bar{a}_{\ell}(i+k)(\ell+j)
      \langle z^{i+k-1},z^{\ell+j-1}\rangle_{L^{2}(\mu_{\alpha})}\nonumber\\
      &=\sum_{i=0}^{t}a_{i}\bar{a}_{i+k-j}(i+k)^{2}\|z^{i+k-1}\|_{L^{2}(\mu_{\alpha})}^2\nonumber\\
      &=\sum_{i=0}^{t}a_{i}\bar{a}_{i+k-j}
      (i+k)\prod_{\ell=1}^{i+k} \frac{\ell}{\ell+\beta}\label{L2InnerProducts}
\end{align}
because $z^l$ and $z^m$ are orthogonal for $l \neq m$ under the
$L^{2}(\mu_{\alpha})$ inner product.

We simplify notation by calling, for $k \in \NN$,
$$\Lambda_\beta(k)= k\prod_{\ell=1}^k\frac{\ell}{\ell+\beta}.$$
  Since
$a_0=1$ and $a_1=-1$, it follows from \eqref{L2InnerProducts} that
\begin{align*}
      \langle (z^{k} f)^{\prime},(z^{k-1}f)^{\prime}\rangle_{L^{2}(\mu_{\alpha})}
      &=-\Lambda_\beta(k),\\
            \langle (z^{k} f)^{\prime},(z^{k}f)^{\prime}\rangle_{L^{2}(\mu_{\alpha})}
      &=\Lambda_\beta(k)+\Lambda_\beta(k+1),\\
      \langle (z^{k} f)^{\prime},(z^{k+1}f)^{\prime}\rangle_{L^{2}(\mu_{\alpha})}
      &=-\Lambda_\beta(k+1),\;\text{ and }\\
      \langle -f^{\prime},(z^{j}f)^{\prime}\rangle_{L^{2}(\mu_{\alpha})}
      &=\left\{
      \begin{array}{ll}
            \Lambda_\beta(1)   &   \text{ if }j=1\\
            0                                   &   \text{ if }j\geq 2.
      \end{array}\right.
\end{align*}
 Thus, in view of \eqref{keySystem}, the
coefficients of $p_{n}$ satisfy the sytem of equations
\begin{align*}
      c_{1}\left[\Lambda_\beta(1)+\Lambda_\beta(2)\right]
      -c_{2}\left[\Lambda_\beta(2)\right]&=\Lambda_\beta(1)\\
      -c_{j-1}\left[\Lambda_\beta(j)\right]+
      c_{j}\left[\Lambda_\beta(j)+\Lambda_\beta(j+1)\right]
      -c_{j+1}\left[\Lambda_\beta(j+1)\right]&=0\\
      -c_{n-1}\left[\Lambda_\beta(n)\right]
      +c_{n}\left[\Lambda_\beta(n)+\Lambda_\beta(n+1)\right]&=0
\end{align*}
or, interpreting $c_{n+1}=0$, equivalently, for all $2\leq j \leq
n+1$:
$$\Lambda_\beta(j)(c_j - c_{j-1}) = \Lambda_\beta(1) (c_1-1).$$
For fixed $k$, $2\leq k\leq n+1$, by a repeated use of the previous
identity, we obtain the following:
\begin{equation}\label{generalCkFormula}
c_{k}=\left[\Lambda_\beta(1) \sum_{j=1}^k
\frac{1}{\Lambda_\beta(j)}\right](c_1-1)+1
\end{equation}

In particular, we have
\[  c_1-1=-\Lambda_\beta(n+1)c_{n} \]
and so, we can recover the value of $c_1$,
$$c_1-1=-\frac{1}{\Lambda_\beta(1)\sum_{j=1}^{n+1}\frac{1}{\Lambda_\beta(j)}}.$$
Finally, we obtain the explicit solution, which can be expressed as
follows, for $1\leq k\leq n$:
\begin{equation}
      c_{k}=\left[\sum_{j=k+1}^{n+1}\frac{1}{j}
            \prod_{\ell=2}^j \left(1+\frac{\beta}{\ell}\right)\right]
            \left[\sum_{j=1}^{n+1}\frac{1}{j}
            \prod_{\ell=2}^j\left(1+\frac{\beta}{\ell}\right)\right]^{-1}.
      \label{CkFormulaWithBetas}
\end{equation}

Alternatively, in the case of the Dirichlet space, we can compute the coefficients $c_k$, $1\leq k \leq n$, using determinants as follows.
For $n\in\NN$,
      let $M_{n}=M$ and $M^{(k)}_{n}=M^{(k)}$ be the $n\times n$ matrices corresponding
      to $f$ as in Theorem \ref{t-LINEAR}.  By Proposition \ref{t-2tdiagonal}, the
      matrix $M_{n}$ is tridiagonal and so it suffices to compute the coefficients
      above and below each entry of its main diagonal.  The
      coefficients in the $j$th column of $M_{n}$ are given by
      \[  \langle (z^{j+\ell} f)^{\prime},(z^{j}f)^{\prime}\rangle_{L^{2}}
              =a_{0}\bar{a}_{\ell}(j+\ell)+a_{1}\bar{a}_{1+\ell}(j+\ell+1)    \]
      where $\ell=-1,0,1$.  Since $a_0=1$ and $a_1=-1$, we obtain
      \begin{align*}
                  \langle (z^{j-1} f)^{\prime},(z^{j}f)^{\prime}\rangle_{L^{2}}
                  &=-j,\\
                  \langle (z^{j} f)^{\prime},(z^{j}f)^{\prime}\rangle_{L^{2}}
                  &=2j+1,\\
                  \langle (z^{j+1} f)^{\prime},(z^{j}f)^{\prime}\rangle_{L^{2}}
                  &=-(j+1),\\
                  \text{ and }\;\langle -f^{\prime},(z^{j}f)^{\prime}\rangle_{L^{2}}
                  &=\bar{a}_{1-j}.
      \end{align*}
      Consequently,
      \[  M_{1}=3,\qquad M_{1}^{(1)}=1,   \]
      \begin{align*}
                  M_{2}&=\left(
                  \begin{array}{cc}
                              3       &   -2\\
                              -2  &   5
                  \end{array}\right),\qquad
                  M_{2}^{(1)}=\left(
                  \begin{array}{cc}
                              1       &   -2\\
                              0       &   5
                  \end{array}\right),\qquad
                  M_{2}^{(2)}=\left(
                  \begin{array}{cc}
                              3       &   1\\
                              -2  &   0
                  \end{array}\right)
      \end{align*}
      \begin{align*}
                  M_{3}=\left(
                  \begin{array}{ccc}
                              3       &   -2  &   0\\
                              -2  &   5       &   -3\\
                              0       &   -3  &   7
                  \end{array}\right),\qquad
                  &M_{3}^{(1)}=\left(
                  \begin{array}{ccc}
                              1       &   -2  &   0\\
                              0       &   5       &   -3\\
                              0       &   -3  &   7
                  \end{array}\right), \, \hdots
      \end{align*}
      \begin{align*}
                  M_{4}=\left(
                  \begin{array}{cccc}
                              3       &   -2  &   0       &   0\\
                              -2  &   5       &   -3  &   0\\
                              0       &   -3  &   7       &   -4\\
                              0       &   0       &   -4  &   9
                  \end{array}\right),\qquad
                  &M_{4}^{(1)}=\left(
                  \begin{array}{cccc}
                              1       &   -2  &   0       &   0\\
                              0       &   5       &   -3  &   0\\
                              0       &   -3  &   7       &   -4\\
                              0       &   0       &   -4  &   9
                  \end{array}\right), \, \hdots
      \end{align*}
      Thus, the optimal approximants to $f$ of orders 1, 2, 3, and 4 are
      \begin{align*}
                  p_{1}(z)&=p_{1}(0)\left(1+\frac{1}{3}z\right),\\
                  p_{2}(z)&=p_{2}(0)\left(1+\frac{5}{11}z+\frac{2}{11}z^2\right),\\
                  p_{3}(z)&=p_{3}(0)\left(1+\frac{13}{25}z+\frac{7}{25}z^2+\frac{3}{25}z^3\right),
                  \;\text{ and }\\
                  p_{4}(z)&=p_{4}(0)\left(1+\frac{77}{137}z+\frac{47}{137}z^2
                  +\frac{27}{137}z^3+\frac{12}{137}z^4\right).
      \end{align*}

What we have shown is that, for any integer $n,$ the optimal approximant for the Dirichlet space is an example of a generalized
Riesz mean polynomial: more specifically, defining
$H_{n}=\sum_{j=1}^{n}\frac{1}{j}$ and $H_0 = 0,$
$$ p_n(z) = p_{n}(0) \left(\sum_{k=0}^{n}\left(1-\frac{H_{k}}{H_{n+1}}\right)z^{k}\right).$$

Moreover, for the Hardy space, the optimal
approximant is a modified Ces\`{a}ro mean polynomial,
$$ p_n(z) =  p_{n}(0)\left(\sum_{k=0}^{n}
                  \left(1-\frac{k+H_{k}}{n+1+H_{n+1}}\right)z^{k}\right),$$
and for the Bergman space, the optimal approximants are
$$ p_n(z) = p_{n}(0)\left(1+\sum_{k=1}^{n}
                  \left(1-\frac{k(k+7)+4H_{k}}{(n+1)(n+8)+4H_{n+1}}\right)z^{k}\right).$$
We will return to these polynomials in Section
\ref{OptimalNormBounds}.

\section{Rate of decay of the optimal norms}\label{OptimalNormBounds}

In this section, we obtain estimates for $\dist_{D_{\alpha}}(1,f\cdot\Pol_{n})$
as $n\rightarrow\infty$, $f\in D_{\alpha}$.  It turns out that the example of $f(z) = 1-z$
in the previous section is a model example for the rate of decay of $\dist_{D_{\alpha}}(1,f\cdot\Pol_{n})$.  We first examine the rate of decay
for this function, then establish such estimates
when $f$ is a polynomial whose zeros are simple and lie in $\CC\setminus\DD$, and then
extend our results to arbitrary polynomials. We conclude with estimates on functions that
admit an analytic continuation to the closed unit disk yet have at least one zero on the circle.

To simplify notation, define the auxiliary function
$\varphi_{\alpha}$ on $[0,\infty)$ to be
\[  \varphi_{\alpha}(s)=\left\{
        \begin{array}{ll}
                    s^{1-\alpha},\;&\text{ if }\alpha<1\\
                    \log^+(s),&\text{ if }\alpha=1.
        \end{array}\right.  \]

\begin{lem}\label{sharp1}
If $f(z) = \zeta - z,$ for $\zeta \in \TT,$ then $\dist_{D_{\alpha}}^{2}(1,f\cdot\Pol_{n})$ is comparable to
$\varphi_{\alpha}^{-1}(n+1)$ for all sufficiently large $n$.
\end{lem}

\begin{proof}
First notice that, for any polynomial $p$ and $\zeta \in \TT$, the polynomial
$q(z)=\zeta p(\zeta z)$ satisfies $\|p(z)(\zeta-z)-1\|_{\alpha}=\|q(z)(1-z)-1\|_{\alpha}$
because rotation by $\zeta$ is an isometry in $D_{\alpha}$.  Therefore, it is enough
to consider the case when $\zeta=1$, i.e. $f(z) = 1 - z$.

Now, recall that by \eqref{CkFormulaWithBetas}, if $f(z)=1-z$, the
optimal approximant of order $n$ to $1/f$ is
\[  p_{n}(z)=p_{n}(0)\sum_{k=0}^{n}c_{k}z^{k},  \]
where
\begin{equation*}
            c_k=\left[\sum_{j=k+1}^{n+1}\frac{1}{j}
            \prod_{\ell=2}^j \left(1+\frac{\beta}{\ell}\right)\right]
            \left[\sum_{j=1}^{n+1}\frac{1}{j}
            \prod_{\ell=2}^j\left(1+\frac{\beta}{\ell}\right)\right]^{-1},
            \;0\leq k\leq n,
\end{equation*}
and $\beta=1-\alpha$.  We claim that $\|p_{n}f-1\|_{\alpha}^{2}$ is comparable
to $\varphi_{\alpha}^{-1}(n+1)$ for all sufficiently large $n$.

First of all, notice that
\[  p_{n}(z)f(z)-1=p_{n}(0)-1+p_{n}(0)
        \left[\sum_{k=1}^{n}(c_{k}-c_{k-1})z^{k}-c_{n}z^{n+1}\right].   \]
To simplify notation, define for $1\leq k\leq n$
\[  a_{k}=c_{k}-c_{k-1}=-
        \left[\frac{1}{k}\prod_{\ell=2}^k \left(1+\frac{\beta}{\ell}\right)\right]
        \left[\sum_{j=1}^{n+1}\frac{1}{j}
        \prod_{\ell=2}^j\left(1+\frac{\beta}{\ell}\right)\right]^{-1}
            \]
and $a_{n+1}=-c_{n}$.  Then
\begin{equation}\label{sumAks}
            \sum_{k=1}^{n}k^{\alpha}|a_{k}|^{2}
            =\left[\sum_{j=1}^{n+1}
            \frac{1}{j}\prod_{\ell=2}^j\left(1+\frac{\beta}{\ell}\right)\right]^{-2}
            \sum_{k=1}^{n}k^{\alpha-2}
            \left[\prod_{\ell=2}^k \left(1+\frac{\beta}{\ell}\right)\right]^{2}.
\end{equation}
Recalling that $2^{-1}x\leq\log(1+x)\leq x$ holds for all $x\in[0,1]$, we see that
\[  \prod_{\ell=2}^k \left(1+\frac{\beta}{\ell}\right)
        =\exp\left[\sum_{\ell=2}^k \log \left(1+\frac{\beta}{\ell}\right)\right]    \]
is comparable to
\[  \exp\left[\beta\sum_{\ell=2}^{k}\frac{1}{\ell}\right],  \]
and so comparable to $k^{\beta}$, when $1\leq k\leq n+1$.  Thus, the sum
in \eqref{sumAks} and
\[  (n+1)^{\alpha}|a_{n+1}|^{2}=(n+1)^{\alpha}\left[\frac{1}{n+1}
        \prod_{\ell=2}^{n+1}\left(1+\frac{\beta}{\ell}\right)\right]^{2}
        \left[\sum_{j=1}^{n+1}\frac{1}{j}
        \prod_{\ell=2}^j\left(1+\frac{\beta}{\ell}\right)\right]^{-2}   \]
are comparable to
\[  \left[\sum_{j=1}^{n+1}\frac{1}{j^{\alpha}}\right]^{-2}
        \sum_{k=1}^{n}\frac{1}{k^{\alpha}}\quad\text{ and }\quad
        \frac{1}{(n+1)^{\alpha}}
        \left[\sum_{j=1}^{n+1}\frac{1}{j^{\alpha}}\right]^{-2}, \]
respectively.  Since $\sum_{j=1}^{n}j^{-\alpha}$ is comparable to
$\varphi_{\alpha}(n+1)$, the sum
\[  \sum_{k=1}^{n+1}k^{\alpha}|a_{k}|^{2}   \]
is comparable to $\varphi_{\alpha}^{-1}(n+1)$ when $n\geq 2$.  This proves the lemma.
\end{proof}

Let us now examine the rate of decay of optimal norms for polynomials whose zeros are simple
and lie in $\CC\setminus\DD$.  To begin, let us introduce some notation.
Let $A(\TT)$ denote the {\it Wiener algebra}, that is, $A(\TT)$ consists of
functions $f$, defined on $\TT$, whose Fourier coefficients are absolutely summable,
and is equipped with the norm
\[\|f\|_{A(\TT)}=\sum_{k=-\infty}^{\infty}|a_k|.\]
The positive Wiener algebra consists of analytic functions whose Fourier coefficients
satisfy $\sum_{k=0}^{\infty}|a_k|<\infty$; in particular, these functions belong to
$H^{\infty}$, the space of bounded analytic functions in $\DD,$ and
$\|f\|_{H^{\infty}}\leq \|f\|_{A(\TT)}$ holds for all $f$ in the positive Wiener algebra, where
$\|f\|_{H^{\infty}}=\sup\{|f(z)|\colon z\in\DD\}$.

\begin{prop}\label{prop43}
Let $\alpha\leq 1$, $t \in\NN$ and $f\in\Pol_{t}$.  If the zeros of $f$ are simple
and lie in $\CC\backslash\DD$, then for each $n>t$ there is $p_{n}\in\Pol_{n}$ such that
$(p_{n}f)(0)=1$,
\begin{equation}\label{distPhiEstimate}
            \|p_{n} f -1 \|^2_\alpha \leq C \varphi_{\alpha}^{-1}(n+1)
\end{equation}
holds for some constant $C$ that depends on $f$ and $\alpha$ but not on $n$, and such that
the sequence $\{p_{n} f\}_{n>t}$ is bounded in $A(\TT)-$norm.
\end{prop}
\begin{proof}
            Suppose $f$ has simple zeros $z_1,\ldots, z_{t}\in\CC\setminus\DD$.  Then there
            are constants $d_1,\ldots,d_{t}$ such that
            \[  \frac{1}{f(z)}=\sum_{j=1}^{t}\frac{d_{j}}{z_{j}-z}
                    =\sum_{k=0}^{\infty}\left(\sum_{j=1}^{t}\frac{d_{j}}{z_{j}^{k+1}}\right)z^{k}.  \]
            Define $b_{k}=\sum_{j=1}^{t}d_{j} z_{j}^{-(k+1)}$ for $k\geq 0$.  It follows that
            the sequence $\{b_{k}\}_{k=0}^{\infty}$ is bounded in modulus by $\sum_{j=1}^{t}|d_{j}|$,
            and the Taylor series representations of $f$ and $1/f$ centered at the origin
            are of the form
            \[  f(z)=\sum_{k=0}^{t}a_kz^k\;\text{ and }\;
                    \frac{1}{f(z)}=\sum_{k=0}^{\infty}b_k z^k,  \]
            for some $a_0,\ldots,a_{t}\in\CC$.  Set $a_{k}=0$ for $k>t$.  Consequently,
            \begin{equation}\label{cancellationLemma}
                        \sum_{j=0}^k b_j a_{k-j}=0\quad\text{ for }k\in\NN \backslash\{0\}.
            \end{equation}

            Consider the polynomial $p_{n}(z) = \sum_{k=0}^n c_k z^k$ with coefficients
            \begin{equation*}
                        c_{0}=a_{0}^{-1}\quad\text{ and }\quad
                        c_k = \left(1-\frac{\varphi_{\alpha}(k)}{\varphi_{\alpha}(n+1)}\right) b_k
                        \text{ for }1\leq k\leq n.
            \end{equation*}
            For convenience of notation, let $c_k = 0$ if $k >n$.  Evidently,
            $(p_{n}f)(0)=1$. Let us prove \eqref{distPhiEstimate}.  To estimate
            $\|p_{n}f-1\|_{\alpha}^{2}$, we consider separately the norms of
            \begin{align*}
                        {\bf mp}&=\sum_{k=t+1}^{n+t}\left(\sum_{i=0}^k c_i a_{k-i}\right)z^k,\;\text{ and }\\
                        {\bf sp}&=\sum_{k=1}^t\left(\sum_{i=0}^k c_i a_{k-i}\right)z^k,
            \end{align*}
            and note that
            \begin{equation}\label{sumOfNorms}
                        \|p_{n}f-1\|_{\alpha}^{2}=\|{\bf mp}\|_{\alpha}^{2}+\|{\bf sp}\|_{\alpha}^{2}
            \end{equation}
            and
            \begin{equation}\label{coeffCk}
                        \sum_{i=0}^k c_i a_{k-i}=
                        \frac{-1}{\varphi_{\alpha}(n+1)}\sum_{i=0}^k \varphi_{\alpha}(i)b_{i}a_{k-i}
            \end{equation}
            by \eqref{cancellationLemma}.  To estimate the norm of {\bf mp}, we need the
            following result.

            \begin{lem}[Control Lemma]\label{fundamental}
            Under the assumptions of Proposition \ref{prop43}, if $k>t$, there is a constant
            $C=C(\alpha,f)$ such that
            \[  \left|\sum_{i=0}^k \varphi_{\alpha}(i) b_i a_{k-i}\right|
                    \leq \frac{C}{(k+1)^{\alpha}}.  \]
            \end{lem}

            We finish the proof of Proposition \ref{prop43} before proving the Control
            Lemma.

            By \eqref{coeffCk} and the Control Lemma \ref{fundamental},
            \begin{align*}
                        \|{\bf mp}\|^2_\alpha
                        &=\sum_{k=t+1}^{n+t} \left|\sum_{i=0}^k c_ia_{k-i}\right|^2 (k+1)^\alpha\\
                        &=\frac{1}{\varphi_{\alpha}^{2}(n+1)}\sum_{k=t+1}^{n+t} \left|\sum_{i=0}^k
                        \varphi_{\alpha}(i)b_{i}a_{k-i}\right|^2 (k+1)^\alpha\\
                        &\leq\frac{C_{1}}{\varphi_{\alpha}^{2}(n+1)}\sum_{k=t+1}^{n+t}
                        \frac{1}{(k+1)^{\alpha}}
            \end{align*}
            for some constant $C_{1}=C_{1}(\alpha,f)$.  It follows now from the estimates
            \[  \sum_{k=t+1}^{n+t} \frac{1}{(k+1)^{\alpha}}\leq\left\{
                    \begin{array}{ll}
                                n(n+t+1)^{-\alpha}\;&\text{ if }\alpha\leq0,\\
                                (1-\alpha)^{-1}[(n+t+1)^{1-\alpha}-(t+1)^{1-\alpha}]&\text{ if }0<\alpha<1,\\
                                \log(n+t+1)-\log(t+1)&\text{ if }\alpha=1
                    \end{array}\right.  \]
            and the elementary inequalities
            \begin{equation*}
                        \begin{array}{ll}
                                    (n+t+1)^{-\alpha}\leq2^{-\alpha}(n+1)^{-\alpha}\;&\text{ for }\alpha\leq0,\\
                                    (n+t+1)^{1-\alpha}\leq2^{1-\alpha}(n+1)^{1-\alpha}&\text{ for }\alpha>0,
                                    \,\text{ and }\\
                                    \log(n+t+1)-\log(t+1)\leq\log(n+1),
                        \end{array}
            \end{equation*}
            that there is a constant $C_{2}=C_{2}(\alpha,f)$ such that
            \begin{equation}\label{sumEstimate}
                        \sum_{k=t+1}^{n+t}\frac{1}{(k+1)^{\alpha}}
                        \leq C_{2}\varphi_{\alpha}(n+1),
            \end{equation}
            and so
            \begin{equation}\label{mpEstimate}
                        \|{\bf mp}\|^2_\alpha\leq\frac{C_{1}C_{2}}{\varphi_{\alpha}(n+1)}.
            \end{equation}

            Next, we estimate the norm of {\bf sp}.  Recalling \eqref{coeffCk}, we see that
            \[  \|{\bf sp}\|^2_\alpha
                        =\frac{1}{\varphi_{\alpha}^{2}(n+1)}\sum_{k=1}^t \left|\sum_{i=0}^k
                        \varphi_{\alpha}(i)b_{i}a_{k-i}\right|^2 (k+1)^\alpha.  \]
            By the Triangle inequality and since $\varphi$ is increasing, if $1\leq k\leq t$, then
            \begin{equation}\label{triangleIneqEstimate}
                        \left|\sum_{i=0}^k \varphi_{\alpha}(i)b_{i}a_{k-i}\right|
                        \leq\|b\|_{\ell^{\infty}}\|a\|_{\ell^{\infty}}(t+1)\varphi_{\alpha}(t),
            \end{equation}
            where  $a=\{a_{k}\}_{k=0}^{\infty}$ and $b=\{b_{k}\}_{k=0}^{\infty}$.  Thus,
            \begin{align*}
                        \|{\bf sp}\|^2_\alpha
                        &\leq\frac{1}{\varphi_{\alpha}^{2}(n+1)}\|b\|_{\ell^{\infty}}^{2}
                        \|a\|_{\ell^{\infty}}^{2}(t+1)^{2}\varphi_{\alpha}^{2}(t)
                        \sum_{k=1}^t (k+1)^{\alpha},
            \end{align*}
            and so
            \begin{equation}\label{spEstimate}
                        \|{\bf sp}\|^2_\alpha\leq\frac{C_{3}}{\varphi_{\alpha}(n+1)}
            \end{equation}
            as $\varphi_{\alpha}(t)\leq\varphi_{\alpha}(n+1)$.
            Hence, \eqref{distPhiEstimate} follows from \eqref{sumOfNorms},
            \eqref{mpEstimate} and \eqref{spEstimate}.

            Finally, we show that the sequence $\{p_{n} f\}_{n>t}$ is bounded in $A(\TT)$.
            Notice that, for $1\leq k\leq t$, \eqref{coeffCk} and \eqref{triangleIneqEstimate}
            imply
            \begin{equation}\label{absSumEstimateSp}
                        \left|\sum_{i=0}^{k}c_{i}a_{k-i}\right|
                        \leq\|b\|_{\ell^{\infty}}\|a\|_{\ell^{\infty}}(t+1)
            \end{equation}
            because $\varphi_{\alpha}(t)\leq\varphi_{\alpha}(n+1)$.  On the other hand,
            for $t<k\leq n+t$, \eqref{coeffCk} and the Control Lemma \ref{fundamental} imply
            \begin{equation}\label{absSumEstimateMp}
                        \left|\sum_{i=0}^{k}c_{i}a_{k-i}\right|
                        \leq\frac{C}{(k+1)^{\alpha}}\varphi_{\alpha}^{-1}(n+1)
            \end{equation}
            for some constant $C=C(\alpha,f)$.  Therefore, by \eqref{absSumEstimateSp},
            \eqref{absSumEstimateMp}, and \eqref{sumEstimate},
            \begin{align*}
                        \|p_{n}f\|_{A(\TT)}&=\sum_{k=1}^{n+t}\left|\sum_{i=0}^{k}c_{i}a_{k-i}\right|\\
                        &\leq\sum_{k=1}^{t}\|b\|_{\ell^{\infty}}\|a\|_{\ell^{\infty}}(t+1)
                        +\frac{C}{\varphi_{\alpha}(n+1)}\sum_{k=t+1}^{n+t}\frac{1}{(k+1)^{\alpha}}\\
                        &\leq\|b\|_{\ell^{\infty}}\|a\|_{\ell^{\infty}}(t+1)t
                        +C\cdot C_{2}
            \end{align*}
            and so $\{p_{n} f\}_{n>t}$ is bounded in $A(\TT)$.  This completes the proof.
\end{proof}

We now proceed to prove Lemma \ref{fundamental}.

\begin{proof}[Proof of Control Lemma \ref{fundamental}]
For $k-t\leq s\leq k$,
\begin{equation}\label{phiPrimeEstimate}
            \varphi^{\prime}_{\alpha}(s)\leq\left\{
            \begin{array}{ll}
                        (1-\alpha)k^{-\alpha}\;&\text{ if }\alpha<0\\
                        (1-\alpha)(k-t)^{-\alpha}&\text{ if }0\leq\alpha<1\\
                        (k-t)^{-1}&\text{ if }\alpha=1.
            \end{array}\right.
\end{equation}
Thus, the Mean Value Theorem, \eqref{phiPrimeEstimate}, and the inequality
\[  (k-t)^{-\alpha}\leq(t+2)^{\alpha}(k+1)^{-\alpha}
        \;\text{ for }\alpha\geq0\text{ and }k\geq t+1, \]
imply that there is a constant $C=C(\alpha,t)$ such that
\[  \varphi_{\alpha}(k)-\varphi_{\alpha}(i)\leq C(k-i)(k+1)^{-\alpha}\;\text{ for }k\geq i. \]
Recalling \eqref{cancellationLemma} and that $a_{i}=0$ for $i>t$, we obtain
\begin{align*}
            \left|\sum_{i=0}^k \varphi_{\alpha}(i) b_i a_{k-i}\right|
            &\leq\left|\sum_{i=0}^k [\varphi_{\alpha}(k)-\varphi_{\alpha}(i)] b_i a_{k-i}\right|\\
            &\leq \sum_{i=k-t}^k [\varphi_{\alpha}(k)-\varphi_{\alpha}(i)]\cdot |b_i a_{k-i}|\\
            &\leq\|a\|_{\ell^{\infty}}\|b\|_{\ell^{\infty}}C(k+1)^{-\alpha}\sum_{i=k-t}^k(k-i),
\end{align*}
where $a=\{a_{i}\}_{i=0}^{\infty}$ and $b=\{b_{i}\}_{i=0}^{\infty}$.  Hence, the
conclusion holds with constant $\|a\|_{\ell^{\infty}}\|b\|_{\ell^{\infty}}Ct(t+1)/2$.
\end{proof}

It seems natural to ask whether the proof  of Theorem \ref{prop43} can
be extended to polynomials $f$ whose zeros are not necessarily simple.   However,
even in the simple case of $f(z)=(1-z)^2$, the coefficients of the Taylor series
representation centered at the origin of $1/f$ are not bounded; consequently, the proof of
Proposition \ref{prop43} cannot be extended directly because the boundedness of these
coefficients is needed.  Nevertheless, if $f$ is an arbitrary polynomial, we can obtain an estimate for $\dist_{D_{\alpha}}(1,f\cdot\Pol_{n})$.
Moreover, using Lemma \ref{sharp1}, we will be able to show this rate of decay is sharp.

\begin{thm}\label{distForPol}
            Let $\alpha\leq 1$.  If $f$ is a polynomial whose zeros lie in
            $\CC\setminus\DD$, then there exists a constant $C=C(\alpha,f)$ such that
            \begin{equation}\label{distEstimate}
                        \dist_{D_{\alpha}}^{2}(1,f\cdot\Pol_{m})
                        \leq C\varphi_{\alpha}^{-1}(m+1)
            \end{equation}
            holds for all sufficiently large $m$.  Moreover, this estimate is sharp in the sense that if
            such a polynomial $f$ has at least one zero on $\TT$, then there exists a constant
            $\tilde{C}= \tilde{C}(\alpha,f)$ such that
      \begin{equation*}
          \tilde{C} \varphi_{\alpha}^{-1}(m+1) \leq \dist_{D_{\alpha}}^{2}(1,f\cdot\Pol_{m}).
      \end{equation*}
\end{thm}
\begin{proof}
Suppose $f$ has factorization
\[  f(z)=K \prod_{k=1}^s (z-z_k)^{r_k}  \]
with $r_{1},\ldots,r_{s}\in\NN$, $z_1,\ldots,z_{s}\in\CC\setminus\DD$ are distinct,
and $K \in\CC\setminus\{0\}$.  Define
\[g(z)=\prod_{k=1}^{s}(z-z_k)\quad\text{ and }\quad
    h(z) =K^{-1}\prod_{k=1}^s (z-z_k)^{\gamma- r_k},\]
where $\gamma=\max\{r_1,\ldots,r_s\}$, and let $d$ equal the degree of $h$.
Then $fh=g^{\gamma}$,
\begin{equation}\label{distEstimateForProduct}
            \dist_{D_{\alpha}}(1,f\cdot\Pol_{n+d})
            \leq\dist_{D_{\alpha}}(1,fh\cdot\Pol_{n})\;\text{ for }n\in\NN,
\end{equation}
and the zeros of $g$ are simple and lie in $\CC\backslash\DD$.

By Proposition \ref{prop43}, for $n>s$, we can choose $q_{n}\in\Pol_{n}$ such
that $(q_{n}g)(0)=1$ and
\begin{equation}\label{qnEstimate}
            \|q_{n}g-1\|^2_\alpha\leq C_{1}\varphi_{\alpha}^{-1}(n+1)
\end{equation}
holds for some $C_{1}=C_{1}(\alpha,g)$,  and such that the sequence
$\{q_{n} g\}_{n>s}$ is bounded in $A(\TT)$.

Let $d\mu_{\alpha}(z)=(1-|z|^{2})^{1-\alpha}\, dA(z)$.  Recalling that
$\|p\|_{\alpha}^2$ is comparable to
$|p(0)|^{2}+D_{\alpha}(p)=|p(0)|^{2}+\|p^{\prime}\|_{L^{2}(\mu_{\alpha})}^{2}$
for all $p\in D_{\alpha}$, we obtain
\begin{align}
            \|q_{n}^{\gamma}g^{\gamma} -1\|^2_\alpha
            &\leq C_{2}\|(q_{n}^{\gamma}g^{\gamma})^{\prime}\|^2_{L^{2}(\mu_{\alpha})}\nonumber\\
            &= C_{2}\|(q_n g)^{\gamma-1}\gamma (q_n^{\prime}g
            +q_n g^{\prime})\|^2_{L^{2}(\mu_{\alpha})}\nonumber\\
            &\leq C_{2}\gamma^{2}\|q_n g\|_{H^{\infty}}^{2\gamma-2}
            \|q_n^{\prime}g+q_n g^{\prime}\|^2_{L^{2}(\mu_{\alpha})}\nonumber\\
            &\leq C_{3}\gamma^{2}\|q_n g\|_{H^{\infty}}^{2\gamma-2}
            \|q_{n}g-1\|^{2}_{\alpha}\nonumber\\
            &\leq C_{3}\gamma^{2}\|q_n g\|_{A(\TT)}^{2\gamma-2}
            \|q_{n}g-1\|^{2}_{\alpha}\label{gGammaEstimate}
\end{align}
for some constants $C_{2}=C_{2}(\alpha)$ and $C_{3}=C_{3}(\alpha)$, as
$(q_{n}g)(0)=1$.  Therefore, \eqref{qnEstimate} and \eqref{gGammaEstimate}
imply that there is a constant $C_{4}=C_{4}(\alpha,\gamma, g)$ such that
\[  \dist_{D_{\alpha}}^{2}(1,g^{\gamma}\cdot\Pol_{n\gamma})
        \leq C_{4}\varphi_{\alpha}^{-1}(n+1)    \]
because $q_{n}^{\gamma}\in\Pol_{n\gamma}$ and $\{q_{n} g\}_{n>s}$ is bounded
in $A(\TT)$.  Thus, by \eqref{distEstimateForProduct},
\begin{equation}\label{distLinear}
            \dist_{D_{\alpha}}^{2}(1,f\cdot\Pol_{n\gamma+d})
            \leq C_{4}\varphi_{\alpha}^{-1}(n+1)\;\text{ when }n>s.
\end{equation}

Let $m>d+(s+1)\gamma$.  Then there exists an integer $a$ and an $n\in\NN$
such that $0\leq a<\gamma$ and $m-d=n\gamma+a$.  In particular, $n>s$ and
\begin{equation}\label{distC4m}
            \dist_{D_{\alpha}}^{2}(1,f\cdot\Pol_{m})
            \leq C_{4}\varphi_{\alpha}^{-1}(n+1)
\end{equation}
follows from \eqref{distLinear} as $m\geq n\gamma+d$.  Finally, the elementary
inequalities
\[  (1+n\gamma+d+a)\leq(\gamma+d)(1+n)\;\text{ and }\;
        (1+n\gamma+d+a)\leq(1+n)^{2\gamma+d}    \]
valid for all $n\in\NN$ imply the existence of a constant $C_{5}=C_{5}(\alpha,\gamma,d)$
such that
\begin{equation}\label{phiC5m}
            \varphi_{\alpha}(m+1)\leq C_{5}\varphi_{\alpha}(n+1).
\end{equation}
Hence, \eqref{distEstimate} holds for $m>d+(s+1)\gamma$ by \eqref{distC4m}
and \eqref{phiC5m}.

Let us now show that the inequality is sharp.  If $f$ is any polynomial with zeros
outside $\DD$ that has at least one zero on $\TT,$ then $f(z) =  h(z)(\zeta - z)$
for some polynomial $h$ of degree say $d.$ Then for any polynomial $p_m$ of degree at most $m,$
$$ \| p_m(z) h(z) (\zeta - z) - 1 \|_{\alpha}^2 \geq \dist_{D_{\alpha}}^{2}(1,(\zeta - z) \cdot\Pol_{m+d}).$$ By Lemma \ref{sharp1}, there exists a constant
$C_1 = C_1(\alpha)$ such that
$$ \dist_{D_{\alpha}}^{2}(1,(\zeta - z) \cdot\Pol_{m+d}) \geq C_1 \varphi^{-1}_{\alpha}(m+d+1).$$ Now, in a manner similar to \eqref{phiC5m}, we can choose
a constant $C_2 = C_2(\alpha, d)$ such that
$$\varphi_{\alpha}^{-1}(m+d+1) \geq  C_2 \varphi_{\alpha}^{-1}(m+1).$$ Finally, letting $\tilde{C} = C_1 C_2$ and noting that the polynomial $p_m$ was arbitrary, we obtain
the desired result that
\begin{equation*}
            \dist_{D_{\alpha}}^{2}(1,f\cdot\Pol_{m}) \geq \tilde{C} \varphi_{\alpha}^{-1}(m+1).\qedhere
\end{equation*}
\end{proof}

In fact, the rates in Theorem \ref{distForPol} hold for more general functions $f$,
namely functions that have an analytic continuation to the closed unit disk.  Since
such functions can be factored as $f(z) = h(z) g(z),$ where $h$ is a polynomial with
a finite number of zeros on the circle and $g$ is a function analytic in the closed
disk with no zeros there, the estimates in Theorem \ref{distForPol} hold for $h$.
Moreover, we can obtain estimates on $g$ that will allow us to give upper bounds on
the product $h(z) g(z)$.  The estimates needed for $g$ are contained in the following lemma.

\begin{lem}\label{estimates_analytic_continuable}
Let $\alpha\leq 1$ and let $g$ be analytic in the closed disk. If $T_n(g)$
 is the Taylor polynomial of degree $n$ of $g$, then
\[\| g-T_n(g)\|^2_\alpha = o\left( \varphi_{\alpha}^{-1}(n+1)\right), \]
and there exists a constant $C= C(\alpha)$ such that
\[\|T_n(g)\|_{M(D_{\alpha})} \leq C. \]
\end{lem}

\begin{proof}
Suppose $g(z) = \sum_{k=0}^{\infty} d_k z^k$ is convergent in the closed unit disk.
Then there exist constants $R> 1$ and $C_1>0$ such that $|d_k| \leq C_1 R^{-k}$.
Therefore

\begin{align*}
\| g-T_n(g)\|^2_\alpha &= \sum_{k = n+1}^{\infty} (k+1)^{\alpha} |d_k|^2\\
&\leq C_1 R^{-2n} \sum_{j = 1}^{\infty} (j+n+1)^{\alpha} R^{-2j}\\
&\leq C_1 R^{-2n} (n+1)^{\alpha} C_2,
\end{align*}
where $C_2=C_2(\alpha, R) = \sum_{j = 1}^{\infty} (j+1)^{\alpha} R^{-2j} $ if
$\alpha \geq 0$ and $C_2 = \sum_{j = 1}^{\infty} R^{-2j} $ if $\alpha < 0$.  In either
case, $C_2 < \infty$ and is independent of $n$.  Therefore, we have that for all $\alpha\leq 1,$
$$ \| g-T_n(g)\|^2_\alpha \leq C_1 C_2 R^{-2n} (n+1)^{\alpha}.$$ Noting that $R^{-2n}$ decays exponentially as $n \rightarrow \infty$ while $\varphi^{-1}_{\alpha}$ decays at a polynomial or logarithmic rate, we obtain that
\[\| g-T_n(g)\|^2_\alpha = o\left( \varphi_{\alpha}^{-1}(n+1)\right). \]

The same type of argument can be used to show that the Taylor polynomials $T_n(g)$
have uniformly bounded multiplier norms. Indeed, if
$f(z) = \sum_{k=0}^{\infty} a_k z^k \in D_{\alpha},$ then in a manner similar to
that above, using the exponential decay of the coefficients $d_k$ of $g$, one can
easily show that for every integer $k,$
$$ \| d_k z^k \cdot f \|_{\alpha} \leq R^{-k}C_{1}C_3 \| f\|_{\alpha},$$
where $C_3 = C_3(k,\alpha) = (k+1)^{\alpha/2}$ if $\alpha \geq 0,$ otherwise, $C_3 = 1$.
Therefore,
$$ \| T_n(g) \cdot f \|_{\alpha} \leq\sum_{k=0}^n \| d_k z^k \cdot f \|_{\alpha}
\leq \left(   \sum_{k = 0}^{n} C_{1}C_3  R^{-k} \right) \| f \|_{\alpha}.$$
Since the series $C =  \sum_{k = 0}^{\infty} C_1 C_3  R^{-k}$ converges,
we obtain
$$ \|T_n(g)\|_{M(D_{\alpha})} \leq C, $$ as desired.
\end{proof}

\begin{thm}\label{OptimalRate}
Let $\alpha\leq 1$.  If $f$ is a function admitting an analytic continuation to the
closed unit disk and whose zeros lie in $\CC\setminus\DD$, then there exists a constant
$C=C(\alpha,f)$ such that
\begin{equation*}
            \dist_{D_{\alpha}}^{2}(1,f\cdot\Pol_{m})
            \leq C\varphi_{\alpha}^{-1}(m+1)
\end{equation*}
holds for all sufficiently large $m$.  Moreover, this estimate is sharp in the sense that if
such a function $f$ has at least one zero on $\TT$, then there exists a constant
$\tilde{C}= \tilde{C}(\alpha,f)$ such that
\begin{equation*}
    \tilde{C} \varphi_{\alpha}^{-1}(m+1) \leq \dist_{D_{\alpha}}^{2}(1,f\cdot\Pol_{m}).
\end{equation*}
\end{thm}

\begin{proof}
Let us first examine the upper bound. Without loss of generality, $f$ is not identically $0,$
and therefore can only have a finite number of zeros on the unit circle $\TT.$
Write $f(z) = h(z) g(z),$ where $h$ is the polynomial formed from the zeros of $f$ that
lie on $\TT,$ and $g$ is analytic in the closed disk with no zeros there.  Therefore,
$1/g$ is also analytic in the closed unit disk (and obviously has no zeros there), and
hence Lemma \ref{estimates_analytic_continuable} applies to $1/g$.  Notice also that $g$
and $g'$ are bounded in the disk, and therefore $g$ is a multiplier for $D_{\alpha}.$

Now, for $m\in\NN$, let $q_m$ be the optimal approximant of order $m$ to $1/h$
and define $p_m = q_m T_m(1/g)$.  By the Triangle Inequality,
\[\|p_m f -1\|_\alpha \leq \|T_m(1/g) g(q_m h -1) \|_\alpha + \|T_m(1/g) g -1\|_\alpha.\]
Since we know that $g$ is a multiplier for $D_{\alpha}$, that the $q_m$ are optimal for $h$,
and that $T_m(1/g)$ are uniformly bounded in multiplier norm by Lemma
\ref{estimates_analytic_continuable}, we see that the square of the first summand on
the right-hand side is dominated by a constant times $\varphi_{\alpha}(m+1)$, for some
constant independent of $m$.  On the other hand, by the second part of Lemma
\ref{estimates_analytic_continuable}, the square of the second summand is
$o(\varphi_{\alpha}(m+1)),$ and thus is negligible by comparison. Therefore,
$$ \dist_{D_{\alpha}}^{2}(1,f\cdot\Pol_{m}) \leq C\varphi_{\alpha}^{-1}(m+1)$$ for
some constant $C=C(\alpha,f),$ as desired.

Let us now address the lower bound for such functions $f.$ Notice first that if the lower
bound holds for functions of the form $(\zeta - z) g(z),$ where $g$ is analytic and without
zeros in the closed unit disk, then the conclusion holds for $f.$
Moreover, as in the proof of Lemma \ref{sharp1}, it is enough to
consider $\zeta = 1.$ Therefore, we write $f(z) = h(z) g(z),$ where
$h(z) = 1-z$ and $g$ as above.  Again, since $g$ is analytic and has
no zeros in the closed disk, note that both $g$ and $1/g$ are
multipliers for $D_{\alpha}.$ Therefore, if $p_m$ is any polynomial
of degree less than or equal to $m,$
$$ \|p_m f - 1\|_{\alpha} \leq \|g\|_{M(D_{\alpha})} \|p_m h - 1/g\|_{\alpha} \leq \|g\|_{M(D_{\alpha})} \|1/g\|_{M(D_{\alpha})} \|p_m f - 1\|_{\alpha}.$$

Now, let's choose $p_m$ to be the optimal approximants of degree less
than or equal to $m$ to $1/f$.  Then by the above discussion, we can
assume $p_m h - 1/g \rightarrow 0$ in $D_{\alpha}$, and in
particular, the norms $\|p_m h \|_{\alpha}$ are bounded. We thus
obtain
\begin{eqnarray*}
     \|p_m f - 1 \|_{\alpha}&  = &  \| p_m h (g - T_m(g)+T_m(g)) - 1\|_{\alpha} \\
     & \geq & \| p_m h T_m(g) - 1\|_{\alpha} - \| p_m h (g - T_m(g))\|_{\alpha}
\end{eqnarray*}
Now, by Lemma \ref{sharp1},
$ \| p_m h T_m(g) - 1\|_{\alpha}^2$ is greater than or equal to a constant times  $\varphi^{-1}_{\alpha}(2m+1), $
which in turn is comparable to $ \varphi^{-1}_{\alpha}(m+1).$ On the other hand,
$$ \| p_m h (g - T_m(g))\|_{\alpha}  \leq \|p_mh\|_{\alpha}\|g - T_m(g)\|_{M(D_{\alpha})},$$ so
by Lemma \ref{estimates_analytic_continuable} and since the norms of $\|p_m h \|_{\alpha}$ are bounded,
this term decays at an exponential rate.  Therefore, there exist constants $C_1$ and $C_2$ such that
$$ \dist_{D_{\alpha}}^{2}(1,f\cdot\Pol_{m}) = \|p_m f - 1 \|_{\alpha}^2 \geq C_1 \| p_m h T_m(g) - 1\|_{\alpha} \geq C_2  \varphi^{-1}_{\alpha}(m+1),$$ as desired.
\end{proof}

\begin{rem}

The methods used in the proofs of Theorems \ref{distForPol} and \ref{OptimalRate}
yield an independent proof of the upper bound for the optimal norm
in the Dirichlet space (the case $\alpha=1$), valid for a class of
functions with the property that the Fourier coefficients of $f$ and
of $1/f$ exhibit simultaneously rapid decay. More specifically, if
$\{a_j\}$ denotes the sequence of Taylor coefficients of a
function $f \in D$ and $\{b_i\}$ denotes the coefficients of
$1/f$, we say that $f$ is a \emph{strongly invertible function} if
$f$ has no zeros in $\DD$ and, if for all $j$ and $k$, we have
$|a_j| \leq C(j+1)^{-3},$ and $|b_k| \leq C(k+1)^{-1},$ for some constant $C$.
For example, one can show that if $f$ is strongly
invertible, then $1/f$ is in the Dirichlet space.
(In fact, much more is true; $1/f \in D_2$.)
That is, strongly invertible implies invertible in
the Dirichlet space, and such functions are known to be cyclic (see
\cite{BS84}, p.~274), and are therefore of interest.
By defining polynomials analogous to those at the end of Section 2, namely,
\[P_n(z)=
\sum_{k=0}^n \left(1-\frac{H_{k}}{H_{n+1}}\right)b_k z^k,\]
one can use the stronger condition on the decay of the coefficients of $1/f$ to prove
a version of the Control Lemma \ref{fundamental} with the coefficients
$H_k$ and then one can obtain the conclusion of Theorem \ref{OptimalRate} for these
strongly invertible functions.  In particular, it can be shown that the
following holds.
\begin{prop}
Let $f$ be a strongly invertible function, $\gamma \in \NN$ and
$g=f^\gamma$. Then there exist polynomials $q_n$ of degree at most $n$ for
which $\|q_n g -1 \|^2_D \leq C/\log (n+2)$.
\end{prop}
\end{rem}

It would be natural to investigate whether these Riesz-type polynomials provide close to
optimal approximants for more general functions, in particular functions of the form $f_{\beta}(z)=(1-z)^{\beta}$, when $\beta<1$.
Another interesting question would be whether the rate of decay that we have observed for functions admitting an analytic continuation
to the closed disk holds for other functions that vanish precisely on the same set.

\section{Logarithmic conditions}\label{log}

It is well-known that if $f$ is invertible in the Hardy or Dirichlet space, then $f$ is cyclic in that space.
In addition, it is easy to see that if both $f$ and $1/f$ are in $D_\alpha$ and
$f$ is bounded then $\log f \in D_\alpha,$ but the converse does not
hold.  The condition that $\log f \in D_\alpha$ can be thought of as an intermediate
between $f \in D_\alpha$ and $1/f \in D_\alpha$. Indeed, $\log f \in D_\alpha$ is
equivalent to $f'/f$ being a $D_{\alpha-2}$ function. On the other hand,  $f \in
D_\alpha$ if and only if $f' \in D_{\alpha-2},$ while $1/f \in D_\alpha$
if and only if $f'/f^2 \in D_{\alpha-2}$.
We therefore want to study the following question:
\begin{prob}\label{prob1}
Is any function $f \in D_\alpha$, with logarithm $q= \log f \in
D_\alpha$, cyclic in $D_\alpha$?
\end{prob}

In several cases the statement is true:
If $\alpha >1$ or $\alpha = 0$, and $f \in D_\alpha$ with its
logarithm $q= \log f \in D_\alpha$, then $f$ is cyclic in
$D_\alpha$.  Indeed, for $\alpha >1$, $\log f \in D_\alpha$ implies $1/f \in
H^\infty$, which is equivalent to the cyclicity of $f$ (see p.~274 of \cite{BS84}). For $\alpha=0$, it is easy to see that if $\log f \in H^1,$
then $\log|f(0)| = (1/2\pi) \int_0^{2\pi} \log|f(e^{i \theta})| d \theta,$ and therefore
 $f$ is outer, that is, cyclic in $H^2$.
Moreover, the logarithmic condition implies the following
interpolation result, valid for all $\alpha <2$.

\begin{lem}\label{interpolation}
Suppose $f \in D_{\alpha}$ and $\log f \in D_{\alpha}$. Then, for any $\tau \in (0,1]$, we have
\[D_{\alpha}(f^{\tau})\leq {\tau}^2\left(D_{\alpha}(f)+D_{\alpha}(\log f)\right),\]
and consequently, $f^{\tau} \in D_{\alpha}$.
\end{lem}
\begin{proof}
It suffices to establish the bound on $D_{\alpha}(f^{\tau})$. To this end, we write
\begin{eqnarray*}
D_{\alpha}(f^{\tau}) &=&\int_{\DD}|(f^{\tau})'(z)|^2d\mu_{\alpha}(z)={\tau}^2\int_{\DD}\left|\frac{f'(z)}{f(z)}\right|^2|f(z)|^{2\tau}d\mu_{\alpha}(z)\\
&=&{\tau}^2\int_{\mathbb{\DD}}\left|\frac{f'(z)}{f(z)}\right|^2|f(z)|^{2\tau}\chi_{\{z\in\DD\colon |f(z)|< 1\}}d\mu_{\alpha}(z)\\
&&+\,{\tau}^2\int_{\DD}\left|\frac{f'(z)}{f(z)}\right|^2|f(z)|^{2\tau}\chi_{\{z\in\DD\colon |f(z)|\geq 1\}}d\mu_{\alpha}(z)\\
&\leq& {\tau}^2\int_{\DD}\left|\frac{f'(z)}{f(z)}\right|^2\chi_{\{z\in\DD\colon |f(z)|< 1\}}d\mu_{\alpha}(z)\\&&+\,
{\tau}^2\int_{\DD}|f'(z)|^2\chi_{\{z\in\DD\colon |f(z)|\geq 1\}}d\mu_{\alpha}(z),
\end{eqnarray*}
and the resulting integrals can be bounded in terms of $D_{\alpha}(f)$ and $D_{\alpha}(\log f)$, as claimed.
\end{proof}
This lemma allows us to show that for a function $f$ in the
Dirichlet space $D$, corresponding to the case $\alpha=1$, the
condition $\log f\in D$ does imply the cyclicity of $f$. The proof
relies on the following theorem due to Richter and Sundberg (see
\cite[Theorem 4.3]{RS92} and let $\mu$ be Lebesgue measure).

\begin{thm}[Richter and Sundberg, 1992]
If $f\in D$ is an outer function, and if $\tau>0$ is such that $f^{\tau}\in D$, then $[f]=[f^{\tau}]$.
\end{thm}

In \cite{RS92}, Richter and Sundberg applied this theorem by showing that if
$f$ is univalent and non-vanishing, then $f^\tau\in D,$ and hence is cyclic.
In what follows, we do not require univalence.

\begin{thm}\label{solvelog}
Suppose $f\in D$ and $\log f \in D$. Then $f$ is cyclic in the Dirichlet space.
\end{thm}
\begin{proof}
As discussed above, the logarithmic condition $\log f\in D$ implies that $f$ is outer. Next, by
Lemma \ref{interpolation}, $f^{\tau} \in D$ for all $\tau>0$, and so $[f]=[f^{\tau}]$ for each $\tau$. Since the Lemma also
implies $f^{\tau}\rightarrow 1$ in $D$ as $\tau\rightarrow 0$, we have $[f]=[1]$, and the assertion
follows.
\end{proof}

The following is the main result for the remaining cases $\alpha <
0$ and $0<\alpha <1$.

\begin{thm}\label{cyclic1}
Let $f \in H^\infty$ and $q=\log f \in D_\alpha$. Suppose there is a
sequence of polynomials $\{q_n\}$ that approach $q$ in $D_\alpha$
norm with $$2 \sup_{z \in \DD} \mathrm{Re}(q(z)-q_n(z)) + \log
(\|q-q_n\|^2_\alpha) \leq C$$ for some constant $C>0$. Then $f$ is
cyclic in $D_\alpha$.
\end{thm}

\begin{rem}
An immediate consequence of Theorem \ref{cyclic1}
is that
if $q = \log f$ can be approximated in $D_{\alpha}$ by polynomials $\{q_n\}$ with $\sup_{z \in \DD}
\mathrm{Re}(q(z)-q_n(z)) < C$, then $f$ is cyclic.
Brown and Cohn proved (see \cite[Theorem B]{BC85}) that for any
closed set of logarithmic capacity zero $E \subset
\partial \DD$, there exists a cyclic function $f$ in $D$ such that
$\mathcal{Z}(f^*) =E$. The functions they build satisfy this hypothesis on $q_n$, and therefore,
these assumptions are always satisfied by
at least one cyclic function, for any potential cyclic function zero set.
\end{rem}

\begin{proof}[Proof of Theorem \ref{cyclic1}]
We can assume $\alpha \leq 1$, because otherwise the statement is
immediate. As discussed earlier in this section, the
function $f$ is in $D_\alpha$. Now, for any sequence of polynomials $p_n,$
by the triangle inequality
\begin{equation}
\|p_n f -1\|_\alpha \leq \|p_n f - e^{-q_n}f\|_\alpha +
\|e^{-q_n} f -1\|_\alpha. \label{trineq}
\end{equation}

The first summand on the right hand side can be bounded by
$$\|(p_n-e^{-q_n})f\|_\alpha \leq \|p_n -
e^{-q_n}\|_{M(D_{\alpha})} \|f\|_\alpha.$$
Moreover, for $\alpha \leq 1,$ the multiplier norm of a
function is controlled by the $H^\infty$ norm of its derivative.
Hence, a good choice of approximating polynomials is to select $\{p_n\}$ so
that $p_n(0)=e^{-q_n(0)}$ and $\|p'_n + q'_n e^{-q_n}\|_{H^\infty} \leq
1/n$, which is possible by Weierstrass' Theorem. With this choice, the first summand on
the right hand side in
\eqref{trineq} approaches $0$ as $n \rightarrow \infty$.

Note that these  polynomials
$p_n$ converge pointwise to $1/f,$ and therefore, to prove the
cyclicity of $f$, it is sufficient to show that the norms of $p_n f-1$  stay bounded.
So what remains to show is that, as $n$ goes to infinity,
$\|e^{-q_n}f-1\|^2_\alpha $ is uniformly bounded. To evaluate this
expression for large $n$, we use the norm in terms of the
derivative:
$$\|e^{-q_n}f-1\|^2_\alpha \approx \|-q'_n e^{-q_n}f + e^{-q_n}f'\|^2_{\alpha-2} + |e^{-q_n(0)}f(0)-1|^2. $$
The last term tends to $0$ since
$q_n$ approaches $q$ pointwise.

In the first summand on the right hand side, taking out a common factor, we
see that
\begin{eqnarray*}
\|-q'_n e^{-q_n}f + e^{-q_n}f'\|^2_{\alpha-2} & \leq &
\|e^{q-q_n}\|^2_{H^{\infty}} \left\|\frac{f'}{f} -
q'_n\right\|^2_{\alpha-2} \\
& \leq &  C e^{2\sup \mathrm{Re}(q-q_n)} \|q - q_n\|^2_{\alpha}
\end{eqnarray*}
for some constant $C$.
Given our assumptions on $q_n$, the right hand side is less than a constant.
This concludes the proof.

\end{proof}

It would be interesting to determine whether the required approximation property of the polynomials $q_n$ in Theorem \ref{cyclic1} is a consequence of
the other hypotheses.

\section{Asymptotic zero distributions for approximating polynomials}
In this paper we have primarily been interested in functions $f\in D_{\alpha}$ that are cyclic and have $f^*(\zeta)=0$ for at least one $\zeta\in \mathbb{T}$. Prime examples of such a function
are
\[f_{\beta}(z)=(1-z)^{\beta}, \quad \beta \in [0,\infty),\] which we have examined closely in this paper for $\beta$ a natural number.

Numerical experiments, described below, suggest that a study of the zero sets $\mathcal{Z}(p_n)$ of approximating polynomials may be interesting from the point of view of cyclicity. It seems that the rate at which zeros approach the circle is related to the extent to which the corresponding polynomials furnish approximants in $D_{\alpha}$. For instance, we have compared the zero sets associated with the Taylor polynomials of $1/f_{\beta}$ with those of Riesz-type polynomials,
\begin{equation}
\mathcal{R}_n\left(\frac{1}{f_{\beta}}\right)(z)=
\sum_{k=0}^n\left(1-\frac{H_{k}}{H_{n+1}}\right)b_kz^k, \quad n\geq
1. \label{rieszpolyszeros}
\end{equation}

\begin{figure}
\includegraphics[width=0.32 \textwidth]{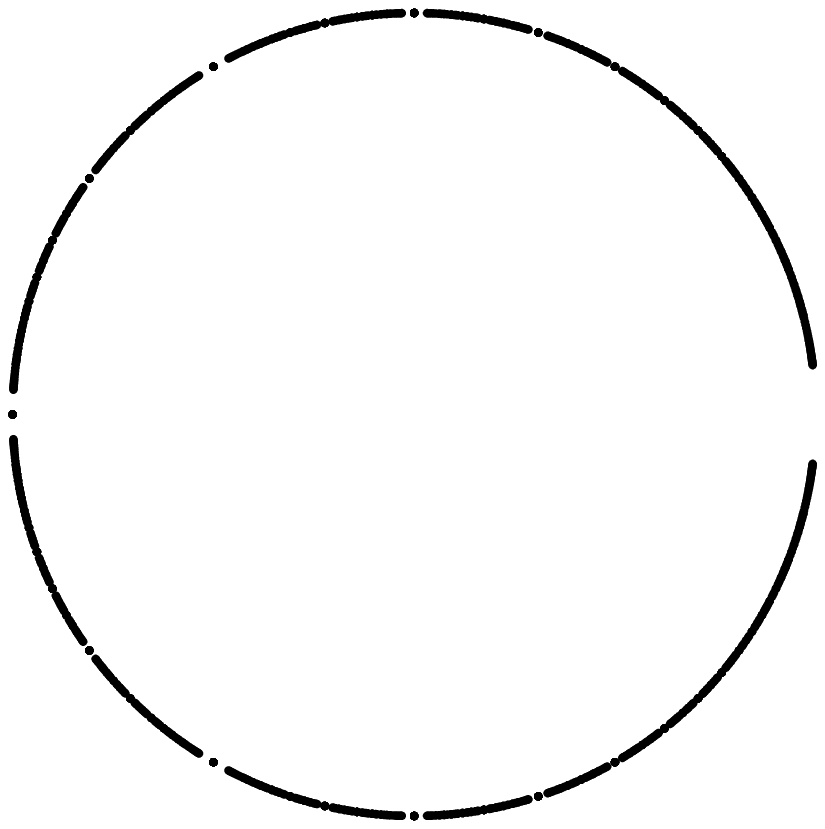}\\
\includegraphics[width=0.32 \textwidth]{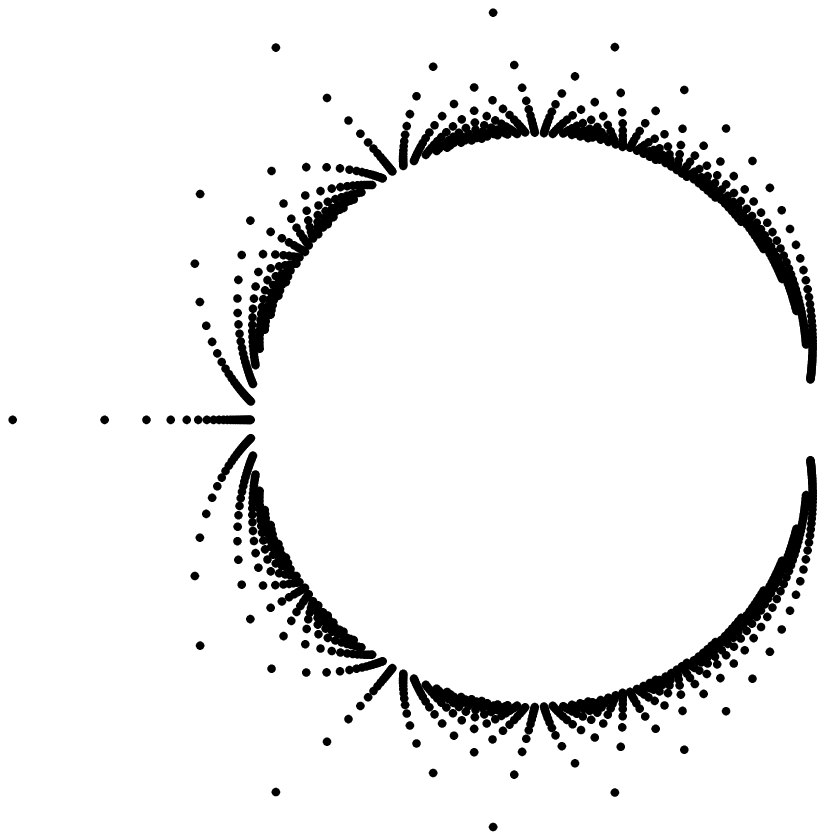}
\includegraphics[width=0.32 \textwidth]{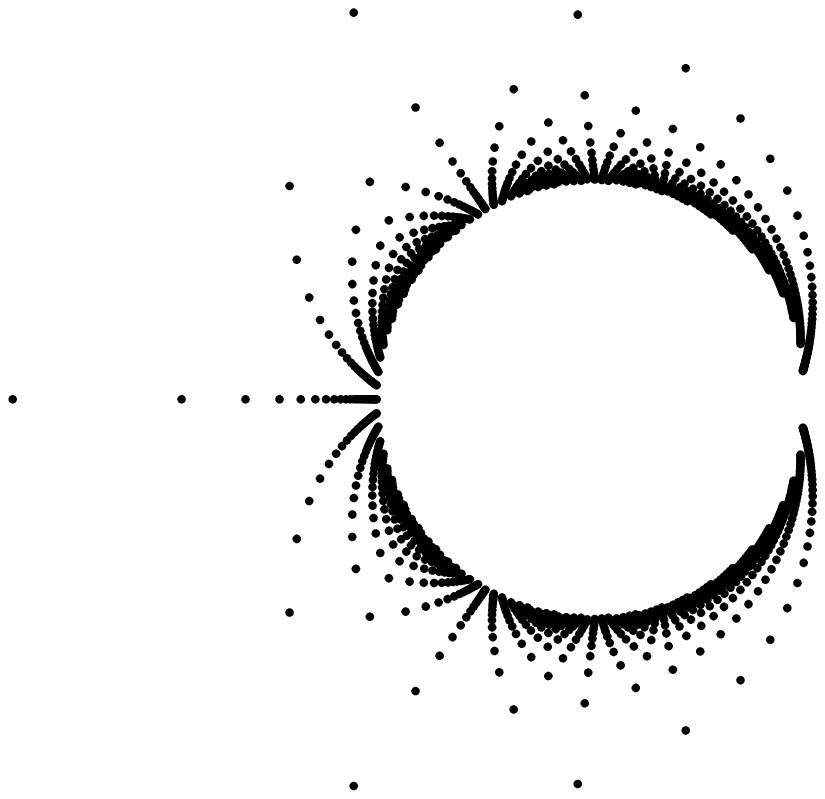}
\caption{Left to right: Successive zero sets $\mathcal{Z}(T_n)$,
$\mathcal{Z}(C_n)$, and $\mathcal{Z}(R_n)$, for $f_1=1-z$, and
$n=1,\ldots,50$.} \label{oneminuszeefig}
\end{figure}

Intuitively, since $1/f_{\beta}$ has a pole at $z=1$, we should expect the approximating polynomials
$p_n$ to be ``large'' in the
intersection of disks of the form $B(1,r)$ with the unit disk. On
the other hand, the remainder functions $p_nf-1$ have to tend to
zero in norm (and hence pointwise). We note that since $1/f_{\beta}$ has a pole on $\TT$, the Taylor
series of $1/f_{\beta}$ cannot have radius of convergence greater than $1$. It therefore follows from
Jentzsch's theorem that every point on $\TT$ is a limit point of the zeros of the sequence $\{T_n(1/f_{\beta})\}_{n=1}^{\infty}$. We refer the reader to \cite{Vargas12} for background material concerning sections of polynomials, and for useful computer code.

We start with the simplest case $f_1(z)=1-z$. The zeros of the Taylor polynomials $T_n(1/f)$, the Ces\`aro polynomials $C_n(1/f)$, and the Riesz polynomials $R_n(1/f)$, for $n=1,\ldots, 50$, can be found in Figure \ref{oneminuszeefig}. All the zeros of these polynomials are located outside the unit disk, and inside a certain cardioid-like curve. In the case of the Taylor polynomials, the explicit formula
\[T_n(1/f_1)(z)=\frac{1-z^{n+1}}{1-z}\]
 holds, and so $\mathcal{Z}(T_ n)$ simply consists of the $n$-th roots of unity, minus the point $\zeta=1$. Replacing Taylor polynomials by Ces\`aro polynomials has the effect of repelling zeros away from the unit circle, and into the exterior of the disk. This effect is even more pronounced for the Riesz polynomials \eqref{rieszpolyszeros}, where it appears that convergence of roots to the unit circumference, and the roots of unity in particular, is somewhat slower. Note also the relative absence of zeros close to
the pole of $1/f_1$, and the somewhat tangential approach region at $\zeta=1$.

Next, we turn to a function with two simple zeros on $\TT$, namely
\[f=1-z+z^2.\] Plots of zeros
of successive approximating polynomials are displayed in Figure \ref{oneminuszeepluszee2fig}. While $\mathcal{Z}(T_n)$ is more complicated, the general features of Figure \ref{oneminuszeefig} persist.
We again note a relative absence of zeros close to the two poles of $1/f$, and the zeros of the Ces\`aro and Riesz polynomials are again located in the exterior disk, and seem to tend to $\TT$ more slowly. We observe
approach regions with vertices at the symmetrically placed poles, and the angle at these vertices seems to decrease as we move from Taylor polynomials through Ces\`aro polynomials to the polynomials in \eqref{rieszpolyszeros}.

It seems natural to suspect that locally the picture would be
similar for a polynomial with a large number of zeros on the unit
circle.

\begin{figure}
\includegraphics[width=0.4 \textwidth]{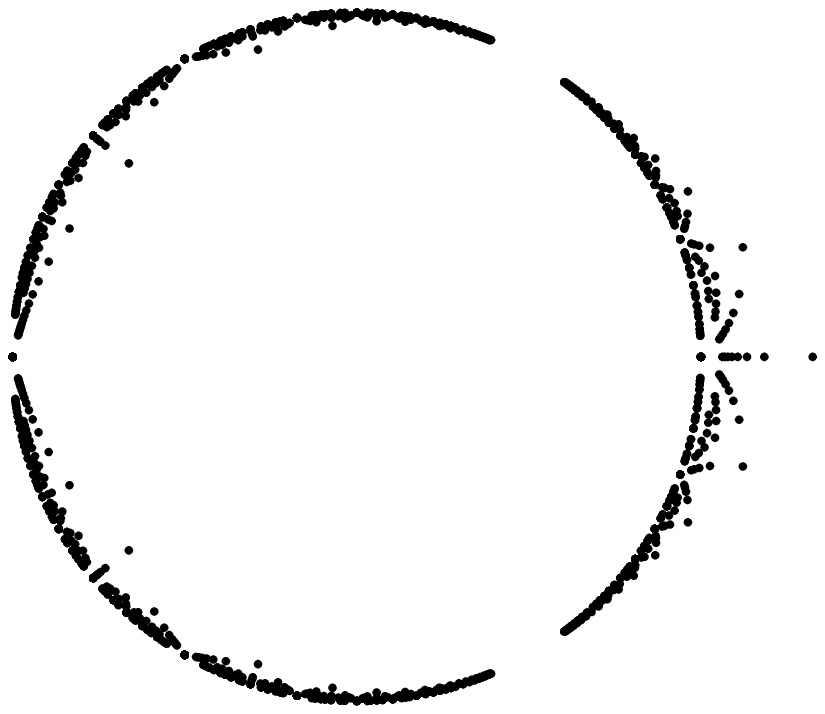}\\
\includegraphics[width=0.45 \textwidth]{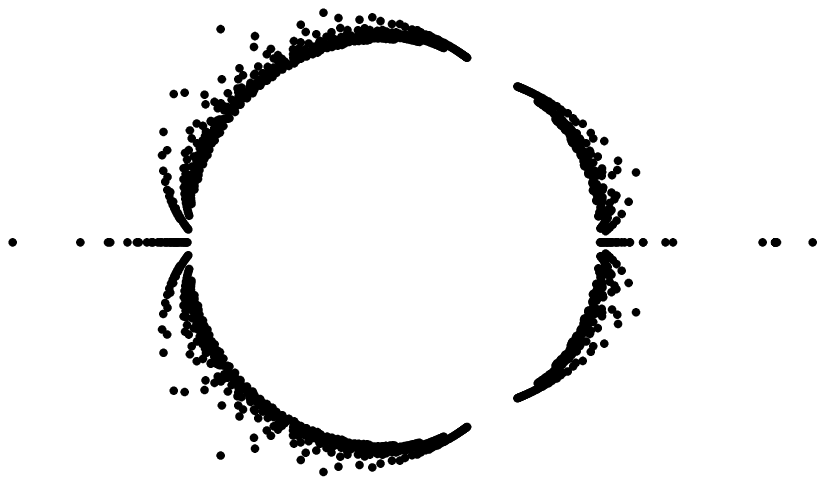}
\includegraphics[width=0.45 \textwidth]{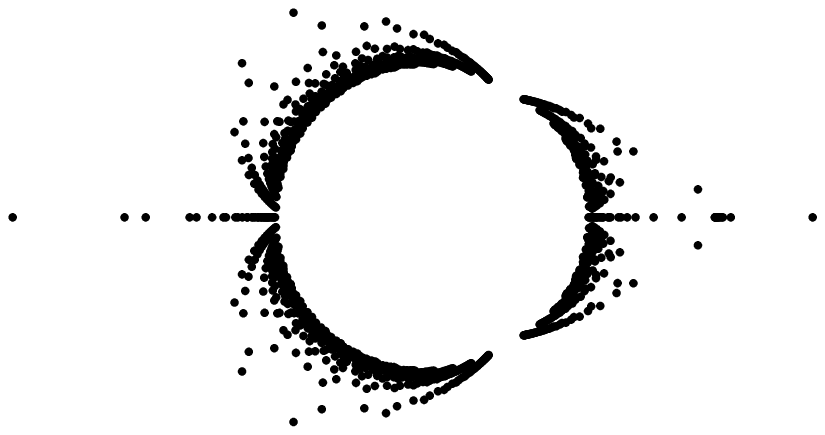}
\caption{Left to right: Successive zero sets $\mathcal{Z}(T_n)$, $\mathcal{Z}(C_n)$, and $\mathcal{Z}(R_n)$, for $f_1=1-z+z^2$, and $n=1,\ldots,50$.}
\label{oneminuszeepluszee2fig}
\end{figure}

It would be interesting to investigate whether there is a relationship between zeros of approximating polynomials, the region of convergence of the Taylor series of $1/f,$
and the cyclicity of $f$ in future work.

\textit{Acknowledgments.} The authors would like to thank Omar El-Fallah, H{\aa}kan Hedenmalm, Dima Khavinson, Artur Nicolau, Boris Shekhtman, and Dragan Vukoti\'{c} for many helpful conversations during the writing of this paper.

\end{document}